\documentclass{article}
\pdfoutput=1
\usepackage[T1]{fontenc}
\usepackage[cp1250]{inputenc}
\usepackage{amsthm}
\usepackage{amsmath}
\usepackage{amsfonts}
\usepackage{multicol}
\usepackage{epsfig}
\usepackage{verbatim}
\usepackage[affil-it]{authblk}
\usepackage[authoryear]{natbib}
\usepackage{amssymb}
\usepackage{pdfpages}
\usepackage{graphicx}
\usepackage{subfig}

\usepackage{hyperref}
\hypersetup{
	bookmarksnumbered,
}%

\newtheorem{theorem}{Theorem}
\newtheorem{lemma}{Lemma}
\newtheorem{corollary}{Corollary}
\theoremstyle{definition}
\newtheorem{example}{Example}

\theoremstyle{plain}

\newcommand\fX{\mathfrak{X}}
\newcommand\R{\mathbb R}
\def\cE{\mathcal E}

\newcommand\cM{\mathcal M}

\newcommand\cS{\mathfrak S}
\newcommand\T{\top}

\newcommand\bc{\mathbf c}
\newcommand\f{\mathbf f}
\newcommand\bu{\mathbf u}
\newcommand\bv{\mathbf v}
\newcommand\x{x}
\newcommand\y{y}
\newcommand\z{\mathbf z}

\newcommand\A{\mathbf A}

\newcommand\E{\mathbf E}
\newcommand\bH{\mathbf H}
\newcommand\I{\mathbf I}
\newcommand\K{\mathbf K}
\newcommand\M{\mathbf M}

\newcommand\Z{\mathbf Z}

\newcommand\0{\mathbf 0}

\newcommand\tr{\mathrm{tr}}

\title{Removal of the points that do not support an E-optimal experimental design}
\author{Radoslav Harman, Samuel Rosa}
\affil{Faculty of Mathematics, Physics and Informatics, Comenius University in Bratislava, Slovakia}
\date{\today} 

\begin{document}
	
\maketitle

\begin{abstract}
	We propose a method of removal of design points that cannot support any E-optimal experimental design of a linear regression model with uncorrelated observations. The proposed method can be used to reduce the size of some large E-optimal design problems such that they can be efficiently solved by semidefinite programming. This paper complements the results of Pronzato [Pronzato, L., 2013. A delimitation of the support of optimal designs for Kiefer's $\phi_p$-class of criteria. Statistics \& Probability Letters 83, 2721--2728], who studied the same problem for analytically simpler criteria of design optimality. 
\end{abstract}

\section{Introduction}

Consider the problem of optimal experimental design (e.g., \cite{Pazman86}, \cite{puk}, \cite{AtkinsonEA07}, \cite{FedorovLeonov}). Let $\fX$ be a finite design space of size $n$ and let the information of a trial under experimental conditions $\x \in \fX$ (in design point $\x$) be expressed by some matrix $\bH(\x) \in \cS^m_+$, the so called elementary information matrix for $x$. The symbol $\cS^m_+$ denotes the set of all $m\times m$ nonnegative definite matrices.  Let $\Xi$ be the set of all approximate designs (i.e., probability measures) on $\fX$; and for any $\xi \in \Xi$, its information matrix is
\begin{equation}\label{eInfMat}
	\M(\xi) = \sum_{\x \in \fX} \xi(\x) \bH(\x).
\end{equation}
Definition \eqref{eInfMat} covers, for instance, the information matrix in the linear regression model $y_i = \f^\T(\x_i)\theta + \varepsilon_i$, $i=1,\ldots,n$, where $\theta \in \R^m$ is the vector of unknown parameters, $\f:\fX \to \R^m$ is the regression function, and the design points $\x_i$ belong to $\fX$. In such a case, the elementary information matrix for $\x \in \fX$ is $\bH(\x)=\f(\x)\f^\T(\x)$. However, this general approach also covers, e.g., the linear regression models with multiple responses in each trial, and it can also be used for optimal augmentation of existing designs as demonstrated in Section 6 of \cite{HarmanTrnovska}. The general form of the problem \eqref{eInfMat} can also be utilized for the construction of constrained optimal designs, see the discussion and references in \cite{Harman14}.

The main result of this paper (Theorem \ref{tMain}) holds also for uncountable compact $\fX$, only $\Xi$ becomes the set of all finitely supported discrete measures on $\fX$, and the sum in \eqref{eInfMat} goes only through $\x$, such that $\xi(\x)>0$. However, for the clarity of the presentation, we work with the discrete case.

For an information function $\Phi:\cS^{m}_+ \to \R$ (see \cite{puk}), a design that maximizes $\Phi(\M(\xi))$ is said to be $\Phi$-optimal. The symbol $\cS^m_+$ denotes the set of all $m\times m$ nonnegative definite matrices. A common class of information functions are the so called Kiefer's $\Phi_p$-optimality criteria for $p \in [-\infty,0]$ (e.g., see \cite{puk}, Chapter 6):
$$
\Phi_p(\M)=
\begin{cases}
\; \Big(\frac{1}{m} \tr(\M^p) \Big)^{1/p}, & p \in (-\infty, 0), \\
\; ( \det(\M) )^{1/m}, & p=0, \\
\; \lambda_1(\M), & p=-\infty,
\end{cases}
$$
for nonsingular $\M$, where $\lambda_1(\M) \leq \ldots \leq \lambda_m(\M)$ are the eigenvalues of $\M$, and $\Phi_p(\M)=0$ if $\M$ is singular.
These criteria include the prominent $D$-, $A$- and $E$-optimality ($p=0,-1,-\infty$, respectively). For simplicity, we assume that there exists a design $\xi$, such that $\M(\xi)$ is nonsingular. It follows that any $\Phi_p$-optimal design is nonsingular (i.e., such that its information matrix is nonsingular).

The performance of algorithms for computing $\Phi$-optimal designs can be improved by iteratively reducing the size of $\fX$. For some criteria $\Phi$ it was shown that any nonsingular design $\xi \in \Xi$ (e.g., $\xi^{(k)}$ obtained in the $k$th iteration of an algorithm) can be used to construct an inequality that must be satisfied by any design point supporting the $\Phi$-optimal design. Therefore, design points not satisfying this inequality can be removed from $\fX$. Generally, the closer the design $\xi$ is to the $\Phi$-optimal design, the more design points can be deleted. Early works pioneering this method for $D$-optimality were \cite{Harman03} and \cite{Pronzato03}. \cite{HarmanPronzato} provided the currently best `deletion method' for $D$-optimality. \cite{Pronzato13} formulated a method for removing design points for $\Phi_p$-optimality criteria, $p\in(-\infty,0)$; thus covering A-optimality ($p=-1$). In this paper, we seek to cover $E$-optimality ($p=-\infty$).

$E$-optimality possesses natural statistical interpretations (it protects against the worst variance of a linear function $\z^\T\widehat{\theta}$ over all $\lVert \z \rVert =1$, cf. \cite[Section 6.4]{puk}; and it minimizes the length of the largest principal axis of the confidence ellipsoid for $\theta$, cf. \cite[Section 2.2.1]{FedorovLeonov}), and is one of the most important optimality criteria (\cite[Section 6.1]{puk}, \cite[Section 10.1]{AtkinsonEA07}). Nevertheless, there is a smaller number of results on $E$-optimality, compared to the $A$ or $D$ criteria. This follows in part from the analytical and computational difficulties in dealing with $E$-optimality: unlike other $\Phi_p$-criteria, it is neither strictly concave nor differentiable. For example, there are only a few algorithms for calculating $E$-optimal designs. Currently, the $E$-optimal designs are usually computed by semidefinite programming (SDP) methods (\cite{VanderbergheBoyd}), by the cutting plane method (\cite{PronzatoPazman}, Section 9.5) or by some general-purpose algorithms of non-differentiable optimization. However, these methods can be efficiently applied only for relatively small sizes of the design space. 
The lack of strict concavity implies that the $E$-optimal information matrix (the information matrix of an $E$-optimal design) is generally not unique. The lack of differentiability makes most optimal design algorithms inapplicable for $E$-optimality, which is a known issue with maximin criteria, see \cite{MandalEA}. Note that $E$-optimality can be viewed as a maximin criterion, because it can be expressed as $\Phi_{-\infty}(\M)=\lambda_1(\M) = \min_{\lVert \bu \rVert = 1} \bu^\T \M \bu$. Because of its importance, there is still a sizeable amount of theoretical results on $E$-optimality, e.g., \cite{PukStudden}, \cite{DetteStudden}, \cite{DetteEA06}, \cite{DetteGrigoriev}.

The size of the design space $\fX$ and the (fixed) number of model parameters determine the dimensionality of the optimization problem. The difficulties with computing $E$-optimal designs make reducing the complexity of the optimization problem by removing unnecessary design points especially useful. Indeed, as will be shown in Section \ref{sEx}, the proposed deletion method allows for applying the known algorithms for $E$-optimality on a larger class of problems. However, the lack of differentiability and strict concavity also means that the deletion method for $E$-optimality requires special attention, as noted by \cite{Pronzato13}. These characteristics of $E$-optimality also lead to a slightly more complicated and less powerful deletion method compared to those for other $\Phi_p$-criteria.

If $\bH(\x)=\f(\x)\f^\T(\x)$, a deletion method for general optimality criterion $\Phi$ is based on the Elfving set $\cE=\mathrm{conv}(\{\f(\x)\}_{\x \in \fX} \cup \{-\f(\x)\}_{\x \in \fX})$ (cf. Theorem 8.5 by \cite{puk}): design points that are not extreme points of $\cE$ can be removed for any information function $\Phi$ without losing any $\Phi$-optimal designs. One may determine if a given $x \in \fX$ can be deleted by checking feasibility of the linear program:
$$\begin{aligned}
&\min_{\alpha \in \R^{n-1}, \beta \in \R^n} (\alpha^\T,\beta^\T)\bc \\
\text{s.t. } &\f(\x) = \sum_{\y \neq \x} \f(y)\alpha_\y - \sum_{\y \in \fX}\f(\y)\beta_\y \\
&\alpha \geq \0_{n-1},\,\, \beta \geq \0_n,\,\,\sum_\y \alpha_\y + \sum_y \beta_\y = 1.
\end{aligned}$$
for arbitrary $\bc \in \R^{2n-1}$, assuming that $\f(y) \neq \f(x)$ for all $y\in\fX$, $y \neq x$.
Note that the deletion method based on Elfving set does not depend on a given design (e.g., it should be performed before an algorithm is run as there is no benefit in using it during the iteration process), and it is rather slow -- to apply this deletion method, one needs to perform feasibility checks for $n$ linear programs.

\section{Necessary condition for support points}

The provided method, as well as those of \cite{HarmanPronzato} and \cite{Pronzato13}, relies on the Equivalence theorem (see \cite{puk}, Chapter 7).
The subgradients of $\Phi_{-\infty}(\M)$ in a nonsingular matrix $\M$ are of the form $\sum_{i=1}^k \alpha_i \bu_i \bu_i^\T$, where $\bu_1, \ldots, \bu_k$ are orthonormal eigenvectors corresponding to $\lambda_1(\M)$ and $\alpha_1,\ldots,\alpha_k$ are some nonnegative weights that sum to 1. Hence, the Equivalence theorem for $E$-optimality on a set of information matrices $\cM$ becomes (see, e.g., \cite{puk}, Theorem 7.21):
\begin{lemma}\label{tET0}
	Let $\xi \in \Xi$, such that $\M(\xi) \in \cM$ is nonsingular. Then $\M(\xi)$ is $E$-optimal in $\cM$ if and only if there exists a nonnegative definite $m \times m$ matrix $\E$ with $\tr(\E)=1$ such that $\tr(\A\E) \leq \lambda_1(\M(\xi))$ for all $\A \in \cM$. In the case of optimality, $\tr(\A\E) = \lambda_1(\M(\xi))$ for any $\A \in \cM$ that is $E$-optimal.
\end{lemma}
In fact, the matrix $\E$ is given by $\sum_{i=1}^k \alpha_i \bu_i \bu_i^\T$ for some weights $\alpha_i$ and eigenvectors $\bu_i$ as described in the previous paragraph.

The Equivalence theorem can be slightly adapted for
$\cM=\{\sum_{\x \in \fX} \xi(\x) \bH(\x) \ \vert \ \xi \in \Xi\}$ considered in this paper.

\begin{corollary}\label{cET}
	Let $\xi \in \Xi$ and let $\M(\xi)$ be nonsingular. Then $\xi$ is $E$-optimal if and only if there exists a nonnegative definite $m \times m$ matrix $\E$ with $\tr(\E)=1$ such that $\tr(\bH(\x)\E) \leq \lambda_1(\M(\xi))$ for all $\x \in \fX$. In the case of optimality, $\tr(\bH(\x^*)\E) = \lambda_1(\M(\xi))$ for any $\x^*$ that supports any $E$-optimal design.
\end{corollary}
\begin{proof}
	Let $\cM=\{\sum_{\x \in \fX} \xi(\x) \bH(\x) \ \vert \ \xi \in \Xi\}$. To prove the first part, it suffices to observe that if $\tr(\bH(\x)\E) \leq \lambda_1(\M(\xi))$ for all $\x \in \fX$, we also have $\tr(\A\E) = \sum_{\x} \tilde{\xi}(\x) \tr(\bH(\x) \E) \leq \lambda_1(\M(\xi))$ for any $\A = \M(\tilde{\xi}) \in \cM$.
	
	Suppose that $\xi$ and $\xi^*$ are $E$-optimal, and that $\xi$ satisfies Lemma \ref{tET0} with $\E$. Then $\tr(\M(\xi^*)\E) = \lambda_1(\M(\xi))$ and $\tr(\bH(\x)\E) \leq \lambda_1(\M(\xi))$ for any $\x$. Then 
	$$\tr(\M(\xi^*)\E) = \sum_{\xi^*(\x)>0} \xi^*(\x) \tr(\bH(\x) \E) \leq \lambda_1(\M(\xi)).$$
	To obtain equality in the last inequality, $\tr(\bH(\x) \E)=\lambda_1(\M(\xi))$ must be satisfied for any $\x$ supporting $\xi^*$.
\end{proof}

The main result of this paper follows.

\begin{theorem}\label{tMain}
	Let $\xi \in \Xi$ be a design with a nonsingular information matrix $\M$ and let $\lambda_1=\lambda_1(\M)$. Take any number, say $s$, of normalized eigenvectors $\bv_1, \ldots, \bv_s$ of $\M$ and any $\alpha_1, \ldots, \alpha_s$, such that $\alpha_i \geq 0$, $\sum_i \alpha_i = 1$, and set 
	$$\Z=\sum_{i=1}^s \alpha_i \bv_i \bv_i^\T \quad\text{and}\quad h=\max_{\x\in\fX} \tr(\bH(\x)\Z).$$
	Then $h \geq \lambda_1$. If $h= \lambda_1$, then $\xi$ is $E$-optimal. If $h>\lambda_1$, then
	$$
	g_h(\x,y) := \sum_{i=1}^m \frac{\bu_i^\T\bH(\x)\bu_i}{(\lambda_i(\M) - h)y + \lambda_1} \geq 1
	$$
	for any $\x$ supporting an $E$-optimal design and for any $y \in [0, \lambda_1 / (h-\lambda_1))$, where $\bu_1, \ldots, \bu_m$ are the orthonormal eigenvectors corresponding to $\lambda_1(\M) \leq \ldots \leq \lambda_m(\M)$.
\end{theorem}

\begin{proof}
First, observe that
\begin{equation}\label{eTrIneq}
	\max_{\x} \tr(\bH(\x)\A) \geq \sum_{\x}\tilde{\xi}(\x)\tr(\bH(\x)\A) = \tr(\M(\tilde{\xi})\A)
\end{equation}
for any $\A \in \cS^m_+$ and any information matrix $\M(\tilde{\xi})$, $\tilde{\xi} \in \Xi$.
Therefore,
$$
h=\max_{\x} \tr(\bH(\x) \Z) \geq \tr(\M\Z) = \sum_{i=1}^s \alpha_i \bv_i^\T \M \bv_i \geq \lambda_1.
$$
Moreover, $\tr(\Z)=1$. Let $h=\lambda_1$. Then the inequalities become equalities; in particular, $\max_\x \tr(\bH(\x) \Z) = \lambda_1$. Therefore, Corollary \ref{cET} yields that $\xi$ is $E$-optimal with $\E=\Z$.

Now, let $h>\lambda_1$ and let $\x^* \in \fX$ support an $E$-optimal design $\xi^*$. Let us denote $\M^* = \M(\xi^*)$ and $\bH_*=\bH(\x^*)$. Then there exists $\E$, such that $\tr(\E)=1$, $\lambda_1 \geq \tr(\M\E)$ and $\lambda_1(\M^*) = \tr(\bH_*\E)$. It follows that, using \ref{eTrIneq},
$$
h=\max_\x \tr(\bH(\x)\Z) \geq \tr(\M^*\Z) = \sum_{i=1}^s \alpha_i \bv_i^\T \M^* \bv_i \geq \lambda_1(\M^*) \geq \tr(\M\E).
$$
Suppose that $0 \leq y < \lambda_1/(h-\lambda_1)$ and let $\K= y\M + (\lambda_1-hy)\I$, where $\I$ denotes the identity matrix. Then, $\K$ is positive definite and
$$
\tr((\K-\bH_*)\E) = y \tr(\M\E) + (\lambda_1-hy)\tr(\E) - \tr(\bH_*\E) \leq yh + (\lambda_1-hy) - \lambda_1 = 0,
$$
because $\lambda_1 \leq \lambda_1(\M^*) \leq \tr(\bH_*\E)$. Thus, 
\begin{equation}\begin{aligned}\label{eKE}
		\tr(\K\E) 
		&\leq \tr( \bH_*\E) 
		= \Vert \bH^{1/2}_* \K^{-1/2} \K^{1/2} \E^{1/2} \Vert^2 
		\leq   \Vert \bH^{1/2}_* \K^{-1/2}  \Vert^2 \cdot \Vert  \K^{1/2} \E^{1/2} \Vert^2 \\
		&= \tr(\bH_*\K^{-1}) \tr(\K \E),
	\end{aligned}\end{equation}
	because the Frobenius norm is sub-multiplicative. The constraints on $y$ and the inequality $\lambda_1 \geq \tr(\M\E)$ guarantee that $\tr(\K \E) > 0$; hence \eqref{eKE} yields $\tr(\bH_*\K^{-1})\geq 1$. The spectral decomposition $\K^{-1} = \sum_{i=1}^m ((\lambda_i(\M) - h)y + \lambda_1)^{-1} \bu_i \bu_i^\T$ then gives $g_h(\x^*,y)\geq 1$.
\end{proof}

Therefore, using any nonsingular design $\xi$, one can remove all design points $\x$ that satisfy $g_h(\x,y)<1$ for some $h$, $y$ given by Theorem \ref{tMain}. For the usual case of the linear regression, where $\bH(\x)=\f(\x)\f^\T(\x)$, we have $h=\max_{\x} \f^\T(\x) \Z \f(\x)$ and
$$
g_h(\x,y) = \sum_{i=1}^m \frac{(\bu_i^\T\f(\x))^2}{(\lambda_i(\M) - h)y + \lambda_1}.
$$

In the following, we assume that we have a design $\xi$ with nonsingular $\M=\M(\xi)$ and $\lambda_1=\lambda_1(\M)$, and that $h>\lambda_1$. One should try to make the values of $g_h(\x,y)$ as low as possible, so that more $\x$'s can be deleted. Fortunately, minimizing $g$ with respect to $y$ is a (one-dimensional) convex problem. Hence, the calculation of optimal $y$ is very fast.

\begin{lemma}
	The function $g_h(\x,y)$ is convex in $y$ on $[0, \lambda_1 / (h-\lambda_1))$.
\end{lemma}
\begin{proof}
	The lemma can easily be proved by calculating the second derivative $\partial^2 g_h(\x,y) /\partial y^2$.
\end{proof}

If the derivative in $y=0$, which is
$$\left.\frac{\partial g_h(\x,y)}{\partial y}\right\vert_{y=0} = \sum_{i=1}^m \frac{\bu_i^\T\bH(\x)\bu_i}{ \lambda_1^2}(h-\lambda_i(\M)) ,$$
is not less than 0, then the optimal $y$ is $y=0$. This is equivalent to
$$
0\leq\sum_{i=1}^m (h-\lambda_i(\M))\tr(\bH(\x) \bu_i\bu_i^\T) 
=h \tr(\bH(\x)) - \tr(\bH(\x)\M).
$$
For example, if $h\geq\lambda_m(\M)$, then $\tr(\bH(\x)\M) \leq \lambda_m(\M) \tr(\bH(\x)) < h \tr(\bH(\x))$ for any $\x \in \fX$, and we always set $y=0$. Otherwise, we seek $y \in [0,\lambda_1 / (h - \lambda_1))$ such that $\partial g_h(\x,y)/\partial y=0.$ This is in fact the problem of finding a root $y$ on $[0,\lambda_1 / (h - \lambda_1))$ of a polynomial of degree at most $2m$.

The following lemma shows that the minimization of $g_h$ implies choosing also $h$ as low as possible. First, denote $G(\x,h):= \min_y g_h(\x,y)$ over $y \in [0, \lambda_1 / (h-\lambda_1))$.

\begin{lemma}
	If $h_1 < h_2$, then $G(\x,h_1) \leq G(\x,h_2)$ for any $\x \in \fX$.
\end{lemma}

\begin{proof}
	The proof is straightforward: by computing the derivative of $g_h(\x,y)$ with respect to $h$ and observing that the set $[0, \lambda_1 / (h-\lambda_1))$ is `increasing' with $h$.
\end{proof}

For a given choice of the eigenvectors $\bv_1, \ldots, \bv_s$, the weights $\alpha=(\alpha_1,\ldots,\alpha_k)^\T$ minimizing $h$ can be obtained by a simple linear program:
\begin{equation}\label{eLP}
	\begin{aligned}
		&\min_{h\in \R, \alpha \in \R^m} h \\
		\text{s.t. } &h\geq \sum_{i=1}^s \alpha_i \bv_i^\T \bH(\x) \bv_i, \quad \x \in \fX \\
		&\sum_{i=1}^m \alpha_i = 1,\,\, \alpha \geq \0_m,
\end{aligned}\end{equation}
where $\0_m$ is the $m\times 1$ vector of zeros. Therefore, the proposed method entails solving one linear program (unlike the deletion method based on the Elfving set, which requires solving $n$ linear programs) and $n$ one-dimensional convex optimizations.

If $\M$ has $m$ distinct eigenvalues, the $m$ normalized eigenvectors are fixed (up to a reflection around origin), so these eigenvectors should be chosen for calculating $h$. However, if some eigenvalue of $\M$ has multiplicity greater than 1, there is freedom in choosing the $\bv_i$'s, but if the minimization \eqref{eLP} is taken also with respect to normalized $\bv_1, \ldots, \bv_s$, the problem becomes nonlinear (and even nonconvex). Therefore, we suggest choosing the set of $m$ orthonormal eigenvectors given by the spectral decomposition of $\M$, with possibly some additional eigenvectors corresponding to $\lambda_1$, if $\lambda_1$ has multiplicity greater than 1. Such a recommendation follows from the fact that $\Z=\sum_i \alpha_i \bv_i \bv_i^\T$ tries to approximate the matrix $\E$ in the Equivalence theorem, which depends on the eigenvectors corresponding to $\lambda_1$. Then, $\alpha$ can be calculated by the linear program \eqref{eLP}.

\section{Example}
\label{sEx}

The $E$-optimality problem is an SDP problem (\cite{VanderbergheBoyd}), which can be solved by the standard solvers like SeDuMi or MOSEK. However, the use of these methods is severely limited by available computer memory, because for a design space of size $\vert \fX \vert = n$, they work with $n \times n$ matrices. For instance, the SDP method can generally calculate $E$-optimal designs for problems of sizes only up to $n\approx10000$ on the computer specified in the next paragraph. Therefore, the proposed deletion method can be used to allow the solvers to deal with larger problems, as demonstrated in Example \ref{exCube}.

All calculations in this section are done in MATLAB on a computer with a 64-bit Windows 8 operating system running an Intel Core i5-4590S CPU processor at 3.00 GHz with 4 GB of RAM; the SDP problems are solved in MATLAB using SeDuMi through the CVX software.  
Throughout, $\bH(\x)=\f(\x)\f^\T(\x)$, in Theorem \ref{tMain} the eigenvectors $\bv_i$ are chosen as the orthonormal set $\bv_1,\ldots,\bv_m$ from the spectral decomposition, $h$ is calculated by the linear program \eqref{eLP} and the $y$'s are calculated by minimizing $g_h(\x,y)$'s.

\captionsetup[subfloat]{captionskip=-10pt}
\begin{figure}[t]
	\centering
	\subfloat[\label{fQuad1}]{\includegraphics[width=0.5\textwidth]{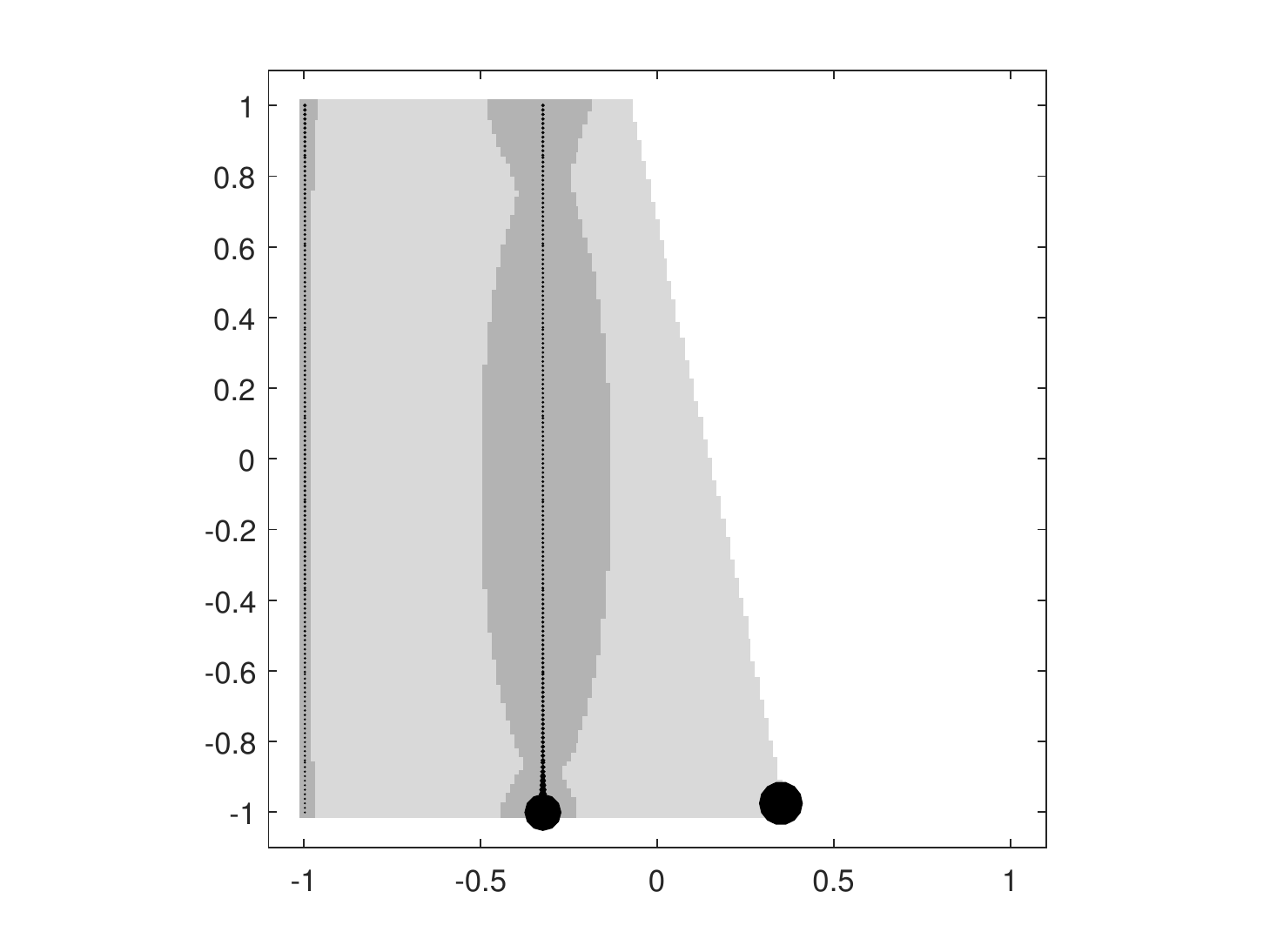}}
	\subfloat[\label{fQuad2}]{\includegraphics[width=0.5\textwidth]{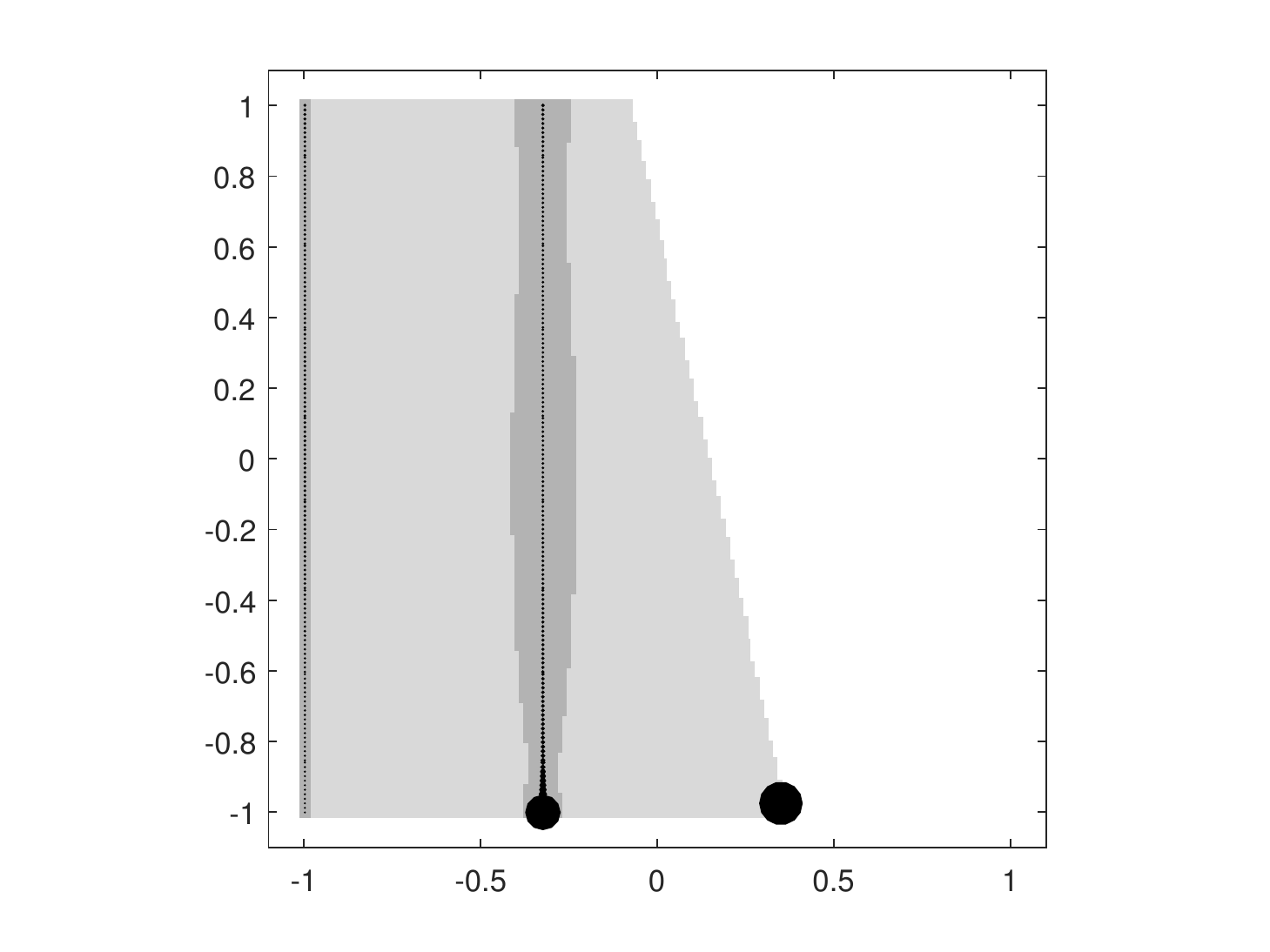}}
	\caption{The deletion method applied for Example \ref{exCube} based on $\tilde{\xi}_1$ (a) and $\tilde{\xi}_2$ (b). The gray area represents $\fX$, i.e., the points $(x_1,x_2)^\T$ that satisfy $x_2 \leq -4.51x_1 + 0.61$. The bright gray region denotes the points that were removed using the proposed method, the points in the dark gray region were not deleted. The black circles represent the values of the obtained $E$-optimal design $\xi^*$, their areas are proportional to the values of $\xi^*$. In each case, the resulting design consists of three sequences of points: for $x_1=-1$, for $x_1=-0.325$ and for $x_1=0.35$. Many of the points in these sequences attain rather small values of $\xi^*$ (smaller than 0.02), but generally they cannot be omitted.}
	\label{fQuad}
\end{figure}
\begin{example}\label{exCube}
	Consider the quadratic regression on the square $[-1,1]^2$ discretized uniformly into $161\times161$ design points. Therefore, the regressors are of the form $\f(x_1,x_2)=(1,x_1,x_2,x_1^2,x_2^2)^\T$ with $x_1,x_2 \in \{\pm k/80 \ \vert\ k=0,\ldots,80\}$. Such models are typically used in the response surface methodology (see, e.g., \cite{MyersEA}). Moreover, suppose that only some combinations of $x_1$ and $x_2$ are allowed in the experiment, expressed by a constraint $x_2 \leq ax_1 + b$; thus, only design points satisfying this constraint belong to $\fX$. For this example, we randomly selected $a=-4.5117$ and $b=0.6091$. To our best knowledge, for the current model on the constrained design space analytical formulas on $E$-optimal designs are not known, unlike for the model on the entire square, which is a rather simple design problem.
	
	Out of the total $n_0=161^2=25921$ original design points, in total $n=14701$ of them satisfy the constraint, and the SDP method runs out of memory while trying to calculate an $E$-optimal design on all $n$ points. However, using the deletion method, this can be remedied. We first calculate an $E$-optimal design $\tilde{\xi}_1$ on a subset $\tilde{\fX}_1$ consisting of $8000$ points chosen at random from $\fX$. Based on $\tilde{\xi}_1$, $n_{\mathrm{del}} =11197$ points (from the entire set $\fX$) are deleted, and on the remaining $n-n_{\mathrm{del}} = 3504$ points an $E$-optimal design can easily be calculated. Theorem \ref{tMain} guarantees that this design is $E$-optimal on the original $14701$-point design space. The deletion results as well as the final $E$-optimal design are illustrated in Figure \ref{fQuad1}.
	
	On the computer specified above, the calculation of $\tilde{\xi}_1$ is performed in around 5 minutes, the deletion method takes approximately 3.3 minutes and the $E$-optimal design on the remaining $3504$ points is calculated in less than 10 seconds.
	
	Instead of selecting a random subset of $\fX$, the removal of unnecessary design points can also be performed based on a design that is $E$-optimal on a less dense grid. For instance, we may halve the density of the discretization by including in $\tilde{\fX}_2$ only $x_1,x_2 \in \{\pm k/40 \ \vert\ k=0,\ldots,50\}$; then $\tilde{\fX}_2$ consists of the $3717$ out of these design points that satisfy the constraint $x_2 \leq a x_1 + b$. Note that $\tilde{\fX}_2$ is indeed a subset of $\fX$. Then, $\tilde{\xi}_2$ that is optimal on $\tilde{\fX}$ is obtained in less than 3 seconds, and based on $\tilde{\xi}$, $12895$ out of the total $14701$ points can be deleted in around 3.5 minutes. An $E$-optimal design on the remaining points (which is also $E$-optimal on the entire $\fX$) can be calculated in less than 2 seconds. The results are illustrated in Figure \ref{fQuad2}.
	
	Note that the deletion method based on the Elfving set takes more than two hours and it does not delete any design points in the current example, although this method generally deletes a nonzero number of points.
\end{example}

The amount of points removed by the proposed method naturally depends on the selected model. Although the deletion method generally allows for solving problems of greater size, the increase in size may be rather small. For instance, in settings identical to Example \ref{exCube}, only with added interaction term (i.e., $\f(x_1,x_2)=(1,x_1,x_2,x_1^2,x_2^2,x_1 x_2)^\T$), the deletion method removes smaller number of points. The approach based on randomly selecting $8000$ points removed $3378$ points from the original $14701$-point $\fX$ for one such random selection. That is not enough to allow one to apply the SDP algorithm to the remaining $11323$ points on the specified computer. However, by using the approach of utilizing the less dense discretization, similarly to Example \ref{exCube}, $5108$ points can be removed. On the remaining $9593$ design points the $E$-optimal design that is also optimal on $\fX$ can be calculated in around 40 minutes. Hence, it seems that considering a slightly less dense discretization for discretized models to delete non-optimal design points may be an efficient approach.

\section*{Acknowledgements}
This work was supported by the Slovak Scientific Grant Agency [grant VEGA 1/0521/16].

\bibliographystyle{plainnat}
\bibliography{rosa.bib}

\begin{thebibliography}{18}
\providecommand{\natexlab}[1]{#1}
\providecommand{\url}[1]{\texttt{#1}}
\expandafter\ifx\csname urlstyle\endcsname\relax
  \providecommand{\doi}[1]{doi: #1}\else
  \providecommand{\doi}{doi: \begingroup \urlstyle{rm}\Url}\fi

\bibitem[Atkinson et~al.(2007)Atkinson, Donev, and Tobias]{AtkinsonEA07}
A.~C. Atkinson, A.~Donev, and R.~Tobias.
\newblock \emph{Optimum experimental designs, with SAS}.
\newblock Oxford University Press, New York, 2007.

\bibitem[Dette and Grigoriev(2014)]{DetteGrigoriev}
H.~Dette and Y.~Grigoriev.
\newblock {E}-optimal designs for second-order response surface models.
\newblock \emph{The Annals of Statistics}, 42:\penalty0 1635--1656, 2014.

\bibitem[Dette and Studden(1993)]{DetteStudden}
H.~Dette and W.~J. Studden.
\newblock Geometry of {E}-optimality.
\newblock \emph{The Annals of Statistics}, 21:\penalty0 416--433, 1993.

\bibitem[Dette et~al.(2006)Dette, Melas, and Pepelyshev]{DetteEA06}
H.~Dette, B.~Melas, and A.~Pepelyshev.
\newblock Local {c}- and {E}-optimal designs for exponential regression models.
\newblock \emph{Annals of the Institute of Statistical Mathematics},
  58:\penalty0 407--426, 2006.

\bibitem[Fedorov and Leonov(2014)]{FedorovLeonov}
V.~V. Fedorov and S.~L. Leonov.
\newblock \emph{Optimal Design for Nonlinear Response Models}.
\newblock CRC Press, Boca Raton, 2014.

\bibitem[Harman(2003)]{Harman03}
R.~Harman.
\newblock A method how to delete points which do not support a {D}-optimal
  design.
\newblock \emph{Tatra Mountains Mathematical Publications}, 26:\penalty0
  59--67, 2003.

\bibitem[Harman(2014)]{Harman14}
R.~Harman.
\newblock Multiplicative methods for computing {D}-optimal stratified designs
  of experiments.
\newblock \emph{Journal of Statistical Planning and Inference}, 146:\penalty0
  82--94, 2014.

\bibitem[Harman and Pronzato(2007)]{HarmanPronzato}
R.~Harman and L.~Pronzato.
\newblock Improvements on removing nonoptimal support points in {D}-optimum
  design algorithms.
\newblock \emph{Statistics \& Probability Letters}, 77:\penalty0 90–94, 2007.

\bibitem[Harman and Trnovsk\'a(2009)]{HarmanTrnovska}
R.~Harman and M.~Trnovsk\'a.
\newblock Approximate {D}-optimal designs of experiments on the convex hull of
  a finite set of information matrices.
\newblock \emph{Mathematica Slovaca}, 59:\penalty0 693--704, 2009.

\bibitem[Mandal et~al.(2015)Mandal, Wong, and Yu]{MandalEA}
A.~Mandal, W.~K. Wong, and Y.~Yu.
\newblock Algorithmic searches for optimal designs.
\newblock In A.~Dean, M.~Morris, J.~Stufken, and Bingham D., editors,
  \emph{Handbook of Design and Analysis of Experiments}. Chapman \& Hall/CRC,
  Boca Raton, 2015.

\bibitem[Myers et~al.(2016)Myers, Montgomery, and Anderson-Cook]{MyersEA}
R.~H. Myers, D.~C. Montgomery, and C.~M. Anderson-Cook.
\newblock \emph{Response surface methodology: process and product optimization
  using designed experiments}, volume~3.
\newblock John Wiley \& Sons, New Jersey, 2016.

\bibitem[P\'azman(1986)]{Pazman86}
A.~P\'azman.
\newblock \emph{Foundation of Optimum Experimental Design}.
\newblock Reidel Publ., Dordrecht, 1986.

\bibitem[Pronzato(2003)]{Pronzato03}
L.~Pronzato.
\newblock Removing non-optimal support points in {D}-optimum design algorithms.
\newblock \emph{Statistics and Probability Letters}, 63:\penalty0 223–228,
  2003.

\bibitem[Pronzato(2013)]{Pronzato13}
L.~Pronzato.
\newblock A delimitation of the support of optimal designs for {K}iefer’s
  {P}hip-class of criteria.
\newblock \emph{Statistics and Probability Letters}, 83:\penalty0 2721–2728,
  2013.

\bibitem[Pronzato and P\'azman(2013)]{PronzatoPazman}
L.~Pronzato and A.~P\'azman.
\newblock \emph{Design of Experiments in Nonlinear Models}.
\newblock Springer, New York, 2013.

\bibitem[Pukelsheim(1993)]{puk}
F.~Pukelsheim.
\newblock \emph{Optimal design of experiments}.
\newblock Wiley, New York, 1993.

\bibitem[Pukelsheim and Studden(1993)]{PukStudden}
F.~Pukelsheim and W.~J. Studden.
\newblock E-optimal designs for polynomial regression.
\newblock \emph{The Annals of Statistics}, 21:\penalty0 402--415, 1993.

\bibitem[Vandenberghe and Boyd(1999)]{VanderbergheBoyd}
L.~Vandenberghe and S~Boyd.
\newblock Applications of semidefinite programming.
\newblock \emph{Applied Numerical Mathematics}, 29:\penalty0 283--299, 1999.

\end{thebibliography}

\end{document}